\documentclass{amsart}
\usepackage{graphicx}

\newtheorem{theorem}[subsection]{Theorem}

\newtheorem{corollary}[subsection]{Corollary}

\newtheorem*{fact}{Theorem}

\theoremstyle{definition}

\theoremstyle{remark}

\title{Minimal coloring number on minimal diagrams for $\mathbb{Z}$-colorable links}

\author{Kazuhiro Ichihara}
\address{Department of Mathematics, 
College of Humanities and Sciences, Nihon University,
3-25-40 Sakurajosui, Setagaya-ku, Tokyo 156-8550, Japan}
\email{ichihara@math.chs.nihon-u.ac.jp}

\author{Eri Matsudo}
\address{Graduate School of Integrated Basic Sciences, Nihon University,
3-25-40 Sakurajosui, Setagaya-ku, Tokyo 156-8550, Japan}
\email{cher16001@g.nihon-u.ac.jp}

\subjclass[2010]{57M25}
\keywords{$\mathbb{Z}$-coloring, link, minimal diagram}
\date{\today}

\begin{document}

\maketitle

%%%%%%%%%%%%%%%%%%%%%%%%%%%%%%%%%%%%%%%%%%%%%%%%%%%%%%%%%%%%%%%%%%%%%%%%

\begin{abstract}
It was shown that any $\mathbb{Z}$-colorable link has a diagram which admits a non-trivial $\mathbb{Z}$-coloring with at most four colors. 
In this paper, we consider minimal numbers of colors for non-trivial $\mathbb{Z}$-colorings on minimal diagrams of $\mathbb{Z}$-colorable links. 
We show, for any positive integer $N$, there exists a $\mathbb{Z}$-colorable link and its minimal diagram such that any $\mathbb{Z}$-coloring on the diagram has at least $N$ colors. 
On the other hand, it is shown that certain $\mathbb{Z}$-colorable torus links have minimal diagrams admitting $\mathbb{Z}$-colorings with only four colors. 
\end{abstract}

%%%%%%%%%%%%%%%%%%%%%%%%%%%%%%%%%%%%%%%%%%%%%%%%%%%%%%%%%%%%%%%%%%%%%%%%

\section{Introduction}

One of the most well-known invariants of knots and links would be the Fox $n$-coloring for an integer $n \ge 2$. 
For example, the tricolorability is much often used to prove that the trefoil is non-trivial. 

Some of links are known to admit a non-trivial Fox $n$-coloring for any $n \ge 2$.
A particular class of such links is the links with 0 determinants. 
(See \cite{IchiharaMatsudo1} for example.) 
For such a link, we can define a $\mathbb{Z}$-coloring as follows, which is a natural generalization of the Fox $n$-coloring. 

Let $L$ be a link and $D$ a regular diagram of $L$. 
We consider a map $\gamma$ from the set of the arcs of $D$ to $\mathbb{Z}$. 
If $\gamma$ satisfies the condition $2\gamma(a)= \gamma(b)+\gamma(c)$ at each crossing of $D$ with the over arc $a$ and the under arcs $b$ and $c$, then $\gamma$ is called a \textit{$\mathbb{Z}$-coloring} on $D$. 
A $\mathbb{Z}$-coloring which assigns the same integer to all the arcs of the diagram is called the \textit{trivial $\mathbb{Z}$-coloring}. 
A link is called \textit{$\mathbb{Z}$-colorable} if it has a diagram admitting a non-trivial $\mathbb{Z}$-coloring. 
Throughout this paper, we often call the integers appearing in the image of a $\mathbb{Z}$-coloring the {\it colors}. 

\medskip

For the Fox $n$-coloring, the minimal number of colors has been studied in details. 
Actually, the minimal numbers of colors for $n$-colorable knots and links behave interestingly, and are often hard to determine. 

On the other hand, for $\mathbb{Z}$-colorable links, the following was shown by the second author in \cite{Matsudo} and by Meiqiao Zhang, Xian’an Jin and Qingying Deng in \cite{ZhangJinDeng} independently, based on the result given in \cite{IchiharaMatsudo2}. 
 
\begin{fact}
The minimal coloring number of any non-splittable $\mathbb{Z}$-colorable link is equal to $4$.
\end{fact} 

Here the \textit{minimal coloring number} of a diagram $D$ of a $\mathbb{Z}$-colorable link $L$ 
is defined as the minimal number of the colors for all non-trivial $\mathbb{Z}$-colorings on $D$, 
and the \textit{minimal coloring number} of $L$ is defined as the minimum of 
the minimal coloring numbers of diagrams representing the link $L$. 
Note that the minimal coloring number of any splittable $\mathbb{Z}$-colorable link is equal to $2$.

\medskip

Now we remark that the proofs in both \cite{Matsudo} and \cite{ZhangJinDeng} are quite algorithmic. 
Therefore the obtained diagrams in the proofs admitting a $\mathbb{Z}$-coloring with four colors are often very complicated. 

In view of this, in this paper, we consider  the minimal coloring numbers of \textit{minimal diagrams} of $\mathbb{Z}$-colorable links,  that is, the diagrams representing the link with least number of crossings. 
Remark that there are many minimal diagrams for each link in general. 

First we show the following. 

\begin{theorem}\label{thm0}
For any positive integer $N$, there exists a non-splittable $\mathbb{Z}$-colorable link with a minimal diagram admitting only $\mathbb{Z}$-colorings with at least $N$ colors. 
\end{theorem}

In fact, the examples are given by families of pretzel links; $P( n, -n, n, -n, \cdots , n , -n)$ with at least 4 strands, $P(-n, n+1, n(n+1))$ with a positive integer $n$. 
These will be treated in Section 2. 

On the other hand, by considering some particular subfamily, as a corollary, we have the following. 

\begin{corollary}\label{cor1}
There exists an infinite family of $\mathbb{Z}$-colorable pretzel links each of which has a minimal diagram admitting a $\mathbb{Z}$-coloring with only four colors. 
\end{corollary}

Also such examples are given by $\mathbb{Z}$-colorable torus links as follows. 

\begin{theorem}\label{thm1}
For even integer $n>2$ and non-zero integer $p$, 
the torus link $T(pn,n)$ has a minimal diagram admitting a $\mathbb{Z}$-coloring with only four colors. 
\end{theorem}

%%%%%%%%%%%%%%%%%%%%%%%%%%%%%%%%%%%%%%%%%%%%%%%%%%%%%%%%%%%%%%%%

\section{Pretzel links}\label{sec:pretzel1}

In this section, we first prove the next theorem. 

\begin{theorem}\label{thm2}
For an even integer $n\geq 2$, the pretzel link $P(n, -n, \cdots, n, -n)$ with at least $4$ strands has a minimal diagram admitting only $\mathbb{Z}$-colorings with $n+2$ colors. 
\end{theorem}

Here a \textit{pretzel link} $P(a_1, \cdots , a_n)$ is defined as a link admitting a diagram consisting rational tangles corresponding to $1/a_1, 1/a_2,\cdots,1/a_n$ located in line. 
See Figure~\ref{p1} for example. 

Since such a pretzel link $P(n, -n, \cdots, n, -n)$ is known to be non-splittable if $n \ge 2$ and the number of strands is at least 4, Theorem~\ref{thm0} immediately follows from Theorem~\ref{thm2}. 

\begin{proof}[Proof of Theorem~\ref{thm2}]

The pretzel link $P(n, -n, \cdots, n, -n)$ is $\mathbb{Z}$-colorable since its determinant is 0 for the link. 
See \cite{DasbachFuterKalfagianniLinStoltzfus} for example. 
Also, in \cite{IchiharaMatsudo2}, an example of a $\mathbb{Z}$-coloring for the link is given. 

Let $D$ be the diagram of $P(n, -n, \cdots, n, -n)$ illustrated in Figure~\ref{p1}. 
This diagram $D$ is a minimal diagram of the link due to the result in \cite{LickorishThistlethwaite}. 
We set the labels $x_1, x_2, \cdots$ of the arcs of $D$ as shown in Figure~\ref{p1}. 
Remark that some of the arcs are labelled in duplicate. 

\begin{figure}[htb]
\begin{center}
\includegraphics[height=5cm,clip]{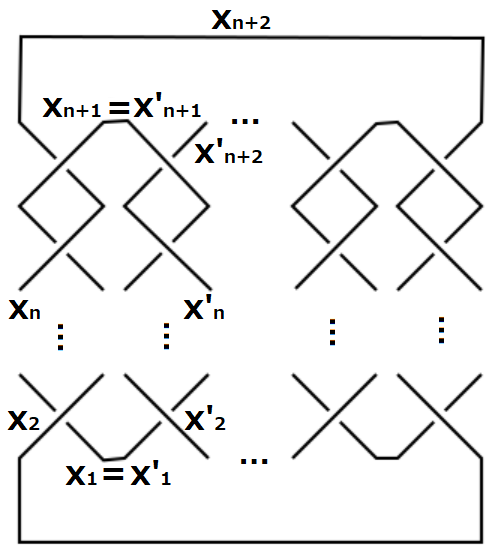}
\caption{}\label{p1}
\end{center}
\end{figure}

Suppose that a non-trivial $\mathbb{Z}$-coloring $\gamma$ is given on $D$.
As shown in \cite{IchiharaMatsudo2}, without changing the number of the colors, 
we may assume that the minimum of the colors for $\gamma$ is $0$, and 
the arcs colored by $0$ cannot cross over the arcs colored by the other colors. 
Thus, due to symmetry, we can set the color $\gamma(x_1)$ of the arc $x_1$ as $0$. 
Also the color $\gamma(x_2)$ of the arc $x_2$ is set as $a>0$, since $\gamma$ is assumed to be non-trivial. 

Here recall that the colors of the arcs under a $\mathbb{Z}$-coloring around a twist, 
i.e., a sequence of bigons in $D$, 
appear as an arithmetic sequence as depicted in Figure~\ref{p2}.

\begin{figure}[htb]
\begin{center}
\includegraphics[height=1.5cm,clip]{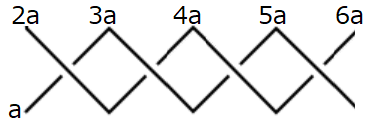}
\caption{}\label{p2}
\end{center}
\end{figure}

Thus, for the arcs $x_1, x_2, \cdots, x_n, x_{n+1}, x_{n+2}$ in $D$, 
we see that $\gamma (x_i) = (i-1)a$ for $ 1 \le i \le n+2$. 
In particular, $\gamma( x_{n+1}) =na$ and $\gamma(x_{n+2}) =(n+1)a$. 

Also, for $x'_1(=x_1), x'_2, \cdots, x'_n, x'_{n+1}(=x_{n+1}), x'_{n+2}$, 
since $\gamma(x'_1) = \gamma (x_1) = 0$ and $\gamma( x'_{n+1} ) = \gamma (x_{n+1}) =na$, 
we see that $\gamma( x'_j ) = (j-1)a$ for $1 \le j \le n+2$. 
In particular, $\gamma( x_{n+1}) =\gamma( x'_{n+1}) =na$. 

In the same way, we can determine all the colors of the arcs on $D$ under $\gamma$. 
It then follows that $D$ is colored by $\gamma$ as illustrated by Figure \ref{p3}. 

\begin{figure}[htb]
\begin{center}
\includegraphics[height=5cm,clip]{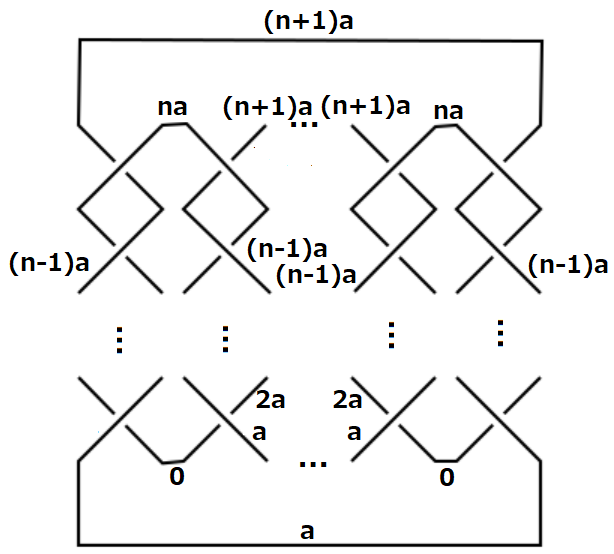}
\caption{}\label{p3}
\end{center}
\end{figure}

As shown in the figure, the colors of $\gamma$ are $0, a, 2a, \cdots, (n+1)a$. 
That is, the minimal coloring number of $D$ is equal to $n+2$.

\end{proof}
\begin{proof}[Proof of Corollary~\ref{cor1}]
We consider the pretzel link $P(2,-2,2,-2,\cdots,2,-2)$. 
The diagram depicted in Figure~\ref{p4} is a minimal diagram by \cite{LickorishThistlethwaite}. 
On the other hand, the $\mathbb{Z}$-coloring given in the figure has only four colors $\{ 0,1,2,3 \}$. 

\begin{figure}[htb]
\begin{center}
\includegraphics[height=3cm,clip]{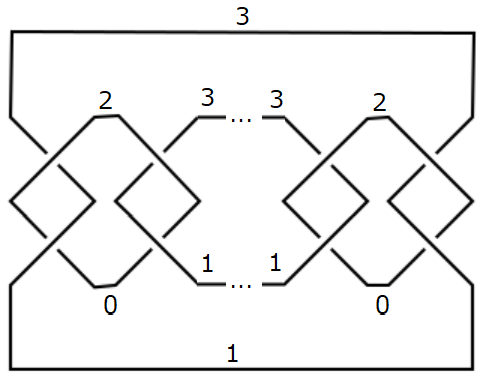}
\caption{}\label{p4}
\end{center}
\end{figure}

\end{proof}

%%%%%%%%%%%%%%%%%%%%%%%%%%%%%%%%%%%%%%%%%%%%%%%%%%%%%%%%%%%%%%%%

Next we consider the pretzel link $P(-n, n+1, n(n+1))$ for an integer $n \ge  2$, and show the following.

\begin{theorem}\label{thm3}
For an integer $n \ge 2$, 
the pretzel link $P(-n, n+1, n(n+1))$ has a minimal diagram admitting only $\mathbb{Z}$-colorings with $n^2+n+3$ colors. 
\end{theorem}

Such pretzel links are all $\mathbb{Z}$-colorable by \cite{DasbachFuterKalfagianniLinStoltzfus} for example. 
In fact, the determinant of  the link $P(-n, n+1, n(n+1))$  is calculated as 
$| (-n) \cdot (n+1) + (-n) \cdot n(n+1) + (n+1)\cdot n(n+1) | = 0 $.  

%\fbox{Remark that the diagram has $n^2+3n+1$ crossings, and so, all the arcs in the diagram are colored by different colors. }

\begin{proof}[Proof of Theorem~\ref{thm3}]
Let $D$ be the diagram illustrated in Figure \ref{p2-1}. 
It is a minimal diagram of $P(-n, n+1, n(n+1))$ as in the previous cases. 
We set the labels $x_1, \cdots, x_{n+2}, x'_1, \cdots, x'_{n+3}, y_1 , \cdots, y_{n^2+n+2}$ of the arcs of $D$ as shown in Figure~\ref{p2-1}. 
Remark that some of the arcs are labelled in duplicate. 

\begin{figure}[htb]
\begin{center}
\includegraphics[height=5cm,clip]{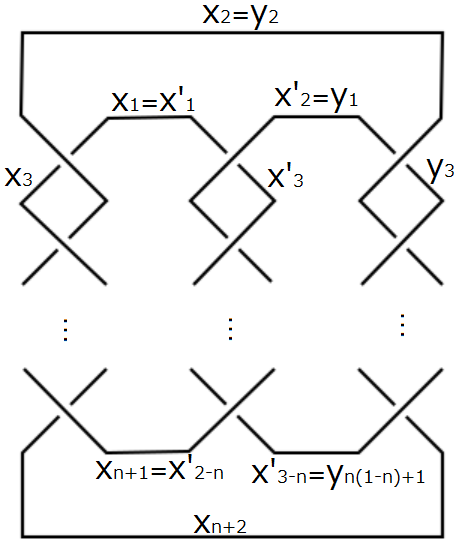}
\caption{}\label{p2-1}
\end{center}
\end{figure}

Suppose that a non-trivial $\mathbb{Z}$-coloring $\gamma$ on $D$ is given. 
As shown in \cite{IchiharaMatsudo2}, without changing the number of the colors, 
we may assume that the minimum of the colors for $\gamma$ is $0$, and 
the arcs colored by $0$ cannot cross over the arcs colored by the other colors. 
Thus, on the diagram $D$, there are only two arcs $x_1$ and $x_{n+2}$ which can be colored by 0. 
We here assume that the color of the arc $x_1$ is $0$, 
since the case that the arc $x_{n+2}$ is colored by $0$ can be treated in a very similar way. 

Set the color of  $x_2$ as $a$ and the color of  $x'_2$ as $b$ in Figure~\ref{p2-2}. 
Note that both $a$ and $b$ have to be non-zero, for $\gamma$ is assumed to be non-trivial. 

Then, for $x_1, x_2, \cdots, x_n, x_{n+1}, x_{n+2}$, 
we see from Figure~\ref{p2} that $\gamma(x_i ) = (i-1)a$ for $ 1 \le i \le n+2$. 
Also, for $x'_1(=x_1), x'_2, \cdots, x'_{n+2} (=x_{n+1}), x'_{n+3}$, 
we see that $\gamma(x'_j) =(j-1)b$ for $ 1 \le j \le n+3$. 
Since $x_{n+1}$ and $x'_{n+2}$ express the same arc, $n a = (n+1) b$ must hold. 
Thus we obtain that $(a, b)=( (n+1) t, n t )$ with some positive integer $t$, for $n$ and $n+1$ are co-prime if $n \ge 2$. 

Now, the color of $y_1=x'_2$ is $b = n t$ and the color of $y_2 = x_2$ is $(n+1)t$. 
It follows that, for the arcs $y_1(=x'_2), y_2(=x_2), y_3, \cdots, y_{n(n+1)+1} (= x'_{n+3}), y_{n(n+1)+2} (=x_{n+2})$, 
we see that $\gamma(y_k) =(n+k-1) t$ for $ 1\le k \le n(n+1)+2$. 
%In particular, $\gamma(y_{n(n+1)+2}) =(n+ n(n+1)+2 -1) t = (n+1)^2 t$ holds. 

We need to check the compatibility of the colors for the arcs labelled in duplicate. 
The colors of $y_{n(n+1)+1}$ is $ ( n + ( n ( n+1) + 1) -1 ) t = n(n+2)t$, which is equal to the color $((n+3)-1)b $ of $x'_{n+3}$. 
The colors of $y_{n(n+1)+2}$ is $ ( n + ( n ( n+1) + 2) -1 ) t = (n+1)^2 t$, which is equal to the color $((n+2)-1)a $ of $x_{n+2}$. 

Consequently, after dividing all the colors by $t$, $\gamma$ is illustrated by Figure~\ref{p2-2}. 
The colors appearing there are 
$\{ 0 , n+1 , 2(n+1), \cdots, n ( n+1), (n+1)^2\} 
\cup
\{ 0 , n, 2n, \cdots, (n+1)n, (n+2)n \} 
\cup
\{ n, n+1, n+2, \cdots , n + n(n+1), n + n(n+1)+1 = (n+1)^2 \}$. 

Note that 
$\{ n+1 , 2(n+1), \cdots, n ( n+1), (n+1)^2\} 
\subset 
\{ n, n+1, n+2, \cdots , n + n(n+1), n + n(n+1)+1 = (n+1)^2 \}$ and 
$\{ n, 2n, \cdots, (n+1)n, (n+2)n \} 
\subset 
\{ n, n+1, n+2, \cdots , n + n(n+1) = (n+2)n, n + n(n+1)+1 = (n+1)^2 \}$. 

Therefore, there are only mutually distinct colors 
$\{ 0 , n, n+1, n+2, n+3, \cdots, n + n(n+1), n + n(n+1)+1 = (n+1)^2 \}$, 
and so, the minimal coloring number of $D$ is $2 + n(n+1)+1 = n^2+n+3$. 
\end{proof}

\begin{figure}[htb]
\begin{center}
\includegraphics[height=5cm,clip]{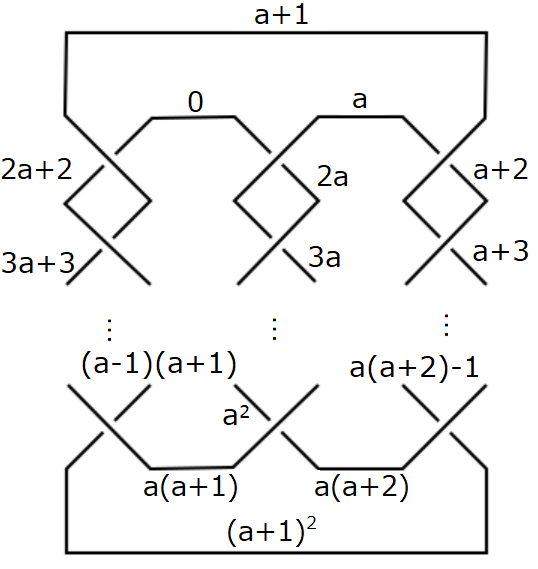}
\caption{}\label{p2-2}
\end{center}
\end{figure}

%%%%%%%%%%%%%%%%%%%%%%%%%%%%%%%%%%%%%%%%%%%%%%%%%%%%%%%%%%%%%%%%

\section{Torus links}\label{sec:torus}

In this section, we consider \textit{torus links}, that is, the links which can be isotoped onto the standardly embedded torus in the 3-space. 
By $T(a,b)$, we mean the torus link running $a$ times meridionally and $b$ times longitudinally.

%\begin{theorem}\label{thm1}
%For even integer $n>2$ and non-zero integer $p$, 
%the torus link $T(pn,n)$ has a minimal diagram admitting a $\mathbb{Z}$-coloring with only four colors. 
%%Whereas $T(p,np)$ admits a minimally colorable diagrams with at most $n (n^2+2)$ crossings. 
%\end{theorem}

\begin{proof}[Proof of Theorem~\ref{thm1}]
Let $D$ be the standard diagram of $T(pn,n)$ illustrated by Figure~\ref{torus}. 
This diagram $D$ has the least number of crossings for the torus link, that is well-known. 
See \cite{Kawauchi} for example. 

\begin{figure}[htb]
\begin{center}
\includegraphics[height=3cm,clip]{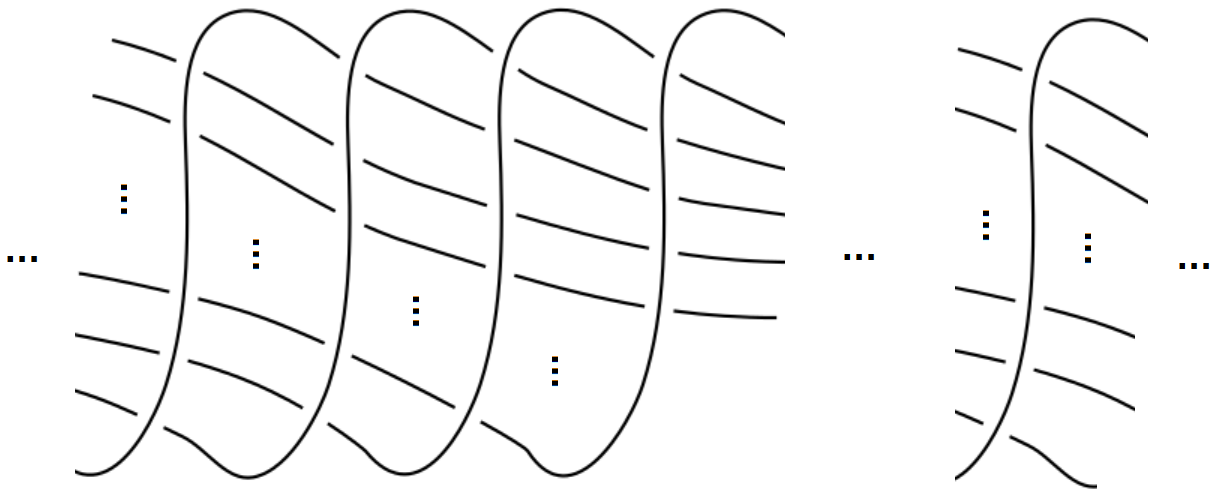}
\caption{}\label{torus}
\end{center}
\end{figure}

In the following, we will find a $\mathbb{Z}$-coloring on $D$ by assigning colors on the arcs of $D$. 

Note that the link has $n$ components running horizontally with $p$ times twistings as shown in $D$. 
In a local view, we see $n$ horizontal parallel arcs in $D$. 
There are $pn$ such sets of parallel arcs in $D$. 
See Figure~\ref{torus1}.

\begin{figure}[htb]
\begin{center}
\includegraphics[height=3cm,clip]{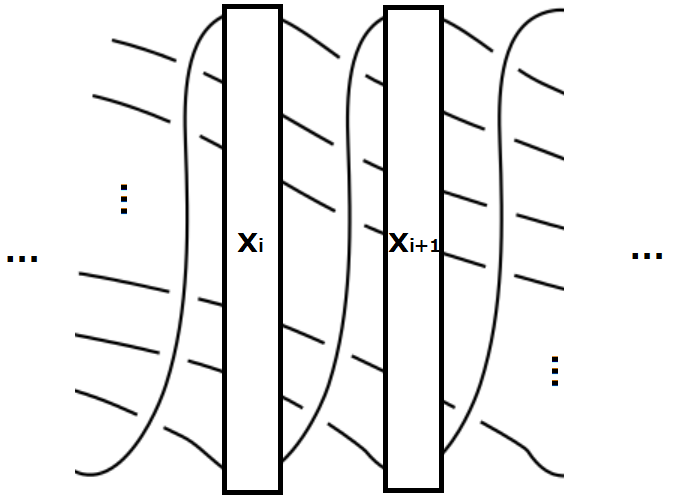}
\caption{}\label{torus1}
\end{center}
\end{figure}

Take a set of such parallel arcs, and assign the colors as $X_1=^t\!\!(1, 0, \cdots, 0, 1)$. 

Then, by considering the condition of the $\mathbb{Z}$-coloring, we can find a matrix 
\[
  A = \left(
    \begin{array}{ccccc}
      0 & \cdots & \cdots & 0 & 1 \\
      -1 & 0 & \cdots & 0 & 2 \\
      0 &  \ddots & \ddots & \vdots & 2 \\
      \vdots & \ddots & \ddots & 0 & \vdots \\
      0 & \cdots & 0 & -1 & 2
    \end{array}
  \right)
\]
such that $A X_1$ gives a set of colors for the right adjacent set of parallel arcs to the prescribed one. 
Set $A X_1$ as $X_2$ and repeat this procedure. 
In fact, we obtain $X_2=AX_1=^t\!\!(1, 1, 2, \cdots, 2)$. 

In the same way, we have $X_4=A^2X_2=^t\!\!(2, 2, 1, 1, 2, \cdots, 2)$ 
as the colors for the 4th set of parallel arcs from the first one. 
By repeating further, we obtain that $X_n=(A^2)^{(n-2)/2}X_2=^t\!\!(2, \cdots, 2, 1, 1)$, since $n$ is even. 
Moreover we see $X_{n+1}=AX_n=^t\!\!(1, 0, \cdots, 0, 1)$. 
That is, $X_1=X_{n+1}$ holds. 

Since there are $pn$ such sets of parallel arcs in $D$, this implies that an appropriate set of colors on all the arcs in $D$, equivalently, a $\mathbb{Z}$-coloring on $D$, can be found. 

Furthermore, we have $X_3=AX_2=^t\!\!(2, 3, 3, \cdots, 2)$, 
$X_5=A^2X_3=^t\!\!(2, 2, 2, 3, 3, 2, \cdots, 2)$, 
and by repeating further, we obtain $X_{n-1}=(A^2)^{(n-2)/2}X_3=^t\!\!(2, \cdots, 2, 3, 3, 2)$.

It concludes that the colors of this $\mathbb{Z}$-coloring is $\{0, 1, 2, 3\}$, that is, the $\mathbb{Z}$-coloring is represented by four colors.
\end{proof}

%%%%%%%%%%%%%%%%%%%%%%%%%%%%%%%%%%%%%%%%%%%%%%%%%%%%%%%%%%%%%%%%

\section*{Acknowledgement}
The authors would like to thank to the anonymous referee for his/her careful reading and pointing out a gap in our proof of Theorem~\ref{thm3} for the previous version of the paper.

%%%%%%%%%%%%%%%%%%%%%%%%%%%%%%%%%%%%%%%%%%%%%%%%%%%%%%%%%%%%%%%%

\end{document}